\documentclass[10pt]{article}
\usepackage{amsmath,amssymb,amsthm,amscd}
\numberwithin{equation}{section}

\def\p{\partial}

\def\i{\sqrt{-1}}

\def\o{\omega}

\def\cC{{\cal C}}

\def\cH{{\cal H}}

\def\h{{\mathfrak h}}

\def\cA{{\mathcal A}}
\def\cC{{\mathcal C}}
\def\cB{{\mathcal B}}

\def\cF{{\mathcal F}}

\def\cH{{\mathcal H}}

\newtheorem{prop}{Proposition}[section]
\newtheorem{theo}[prop]{Theorem}
\newtheorem{lemma}[prop]{Lemma}
\newtheorem{cor}[prop]{Corollary}
\newtheorem{rem}[prop]{Remark}

\newtheorem{defi}[prop]{Definition}

\let\lra=\longrightarrow

\def\mapright\#1{\,\smash{\mathop{\lra}\limits^{\#1}}\,}

\begin{document}
\bibliographystyle{plain}
\title{The Calabi flow on K\"ahler Surfaces with bounded Sobolev constant (I)}
\author{X. X. Chen\footnote{The author is partially supported by NSF.}\; and \;
W. Y.  He\footnote{The author is partially supported by a PIMS postdoc fellowship.}}
\date{}
\maketitle
\tableofcontents

\section{Introduction}
In \cite{Calabi82}, E. Calabi proposed to deform a K\"ahler metric
in the direction of the Levi-Hessian of its scalar curvature. This is called the Calabi flow and it
is a fourth order fully nonlinear parabolic equation aiming to
attack the existence of constant scalar curvature metric (cscK
metric) in a fixed  K\"ahler class.  Note that a cscK metric is a
fixed point of the Calabi flow while an extremal K\"ahler (extK)
metric is a soliton-type solution.  A K\"ahler metric is called extremal if the
complex gradient vector of its scalar curvature is holomorphic.  In particular a K\"ahler-Einstein metric is an extK metric.
The existence of extK metric is one of the central problems in K\"ahler
geometry. The stability conjecture (Yau-Tian-Donaldson conjecture) in K\"ahler geometry asserts that the existence of extK metric  is equivalent to the stability of the underlying K\"ahler structure in geometric invariant theory. However, a cscK metric (or extK) satisfies a fourth order nonlinear partial differential equation. It is hard to attack its existence
problem directly. The study of the Calabi flow  seems to be an
effective method. In \cite{Dona02}, Donaldson described  the limit behavior of the Calabi flow in a conceptual level, relevant to solving the stability conjecture in K\"ahler geometry. The first named author conjectured earlier
\cite{chen05} that the flow exists globally for any smooth initial
K\"ahler metric. However there are rather few analytic techniques which can applied directly as the Calabi flow is a fourth order parabolic equation.
On Riemann surfaces, P. Chrusci\'el showed that the Calabi flow exists for all time and converges to a constant Gaussian curvature metric, making use of the existence of such a metric and  the Bondi mass in general relativity. Without using the Bondi mass, Chen \cite{chen01},  Struwe \cite{Struwe} gave a direct  proof of Chrusci\'el's theorem independently.  While in higher dimensions, the nonlinearity of the Calabi flow becomes more acute and the analytic difficulty becomes more daunting. But recently some progress has been made \cite{chen-he, tw, szekelyhidi}  etc.

In an earlier paper \cite{chen-he}, the authors proved that the main obstruction
to the global existence of the Calabi flow is
the bound of Ricci curvature. In particular the curvature blows up if the flow does not exist for all time. 
It is then natural to study the formation of singularities along the Calabi flow when the curvature blows up. We consider this problem by assuming that  the Sobolev constants of the evolved metrics along the Calabi flow are uniformly bounded. 
The first result of this paper is  
\begin{theo}\label{1} Let $(M, [\omega], J)$ be a compact K\"ahler
manifold.  If the Calabi flow on $(M, [\omega], J)$ exists for time
$t\in [0, T)$ with uniformly bounded Sobolev constants and the curvature tensors become unbounded when $t\rightarrow \infty$, then
there exists a sequence of points $(x_i, t_i) \in (M, [0, T))$
where $\displaystyle \lim_{i\rightarrow \infty} t_i = T\;$ and
$\displaystyle \lim_{i\rightarrow\infty} |Riem(x_i, t_i)|
=Q_i\rightarrow \infty$ such that the pointed manifolds \[(M,
x_i,Q_i g (t_i+t/Q_i^2), J)\] converge locally smoothly to an ancient solution of the Calabi flow
 \[(M_\infty, x_\infty, g_\infty(t), J_\infty),  t\in (-\infty,
0].\] In particular, if $(M, J)$ is a K\"ahler surface, for any $t\in (-\infty, 0]$, $g_\infty(t)\equiv g_\infty(0)$  and $(M_\infty, g_\infty(0), J_\infty)$ is a scalar flat ALE K\"ahler surface with finite total energy. 
\end{theo}

For simplicity,  we call $(M_\infty, g_\infty)$ a {\it maximal bubble} along the Calabi flow. Theorem \ref{1} asserts that a {\it maximal bubble} along the Calabi flow on a K\"ahler surface has to be scalar flat with asymptotically locally Euclidean (ALE) structure at infinity. This property provides strong restrictions on the singularities that might form along the Calabi flow. See Section 4 for more details. 

Scalar flat ALE K\"ahler surfaces are studied extensively in literature, for instance,
Kronheimer \cite{Kronheimer1},  Joyce \cite{joyce},
LeBrun\cite{lebrun01, lebrun02} and
Calderbank-Singer\cite{calsinger} etc. In particular, Kronheimer \cite{Kronheimer1, Kronheimer2} gave a complete classification in hyper-K\"ahler setting. In general the classification of scalar flat ALE K\"ahler surfaces seems to be hard.  But the first named author believes that a scalar flat ALE K\"ahler surface is uniquely determined by its underlying geometric structure.  However, a proof seems out of reach at this point.

Recently Chen-LeBrun-Weber  \cite{chen-lebrun-weber}  studied the existence
of extK metrics in some K\"ahler classes on
$M\sim\mathbb{CP}^2 \sharp 2 \overline{\mathbb{CP}^2};$ $M$ can be obtained by  $\mathbb{CP}^2$ blown up at two different points.  They constructed  an Einstein metric on $M\sim\mathbb{CP}^2 \sharp 2 \overline{\mathbb{CP}^2}$ which is conformal to an extK metric in a particular K\"ahler class. The strategy in \cite{chen-lebrun-weber}  is to use continuous deformation of extK metrics and a weak compactness theorem for extK metrics \cite{chen-weber}.  A possible bubble in the deformation turns out to be a scalar flat ALE K\"ahler surface. A key step is to eliminate the formation of such a bubble by using the smallness of the Calabi energy.  Theorem
\ref{1} makes it possible for us to adopt the strategy
in \cite{chen-lebrun-weber}, {\it albeit} new difficulty arises
in our case. For instance the bound of scalar curvature plays a key role in
\cite{chen-lebrun-weber} while along the Calabi flow, such a bound is not known. Also we need to bound the Sobolev constants along the Calabi flow. Note that this is a non-collapsing (or local non-collapsing) result. In general it is hard to obtain non-collapsing result (or local non-collapsing) along a geometric evolution equation, such as Perelman's non-local collapsing result along the Ricci flow \cite{perelman}. On K\"ahler surfaces with $c_1>0$, we can actually bound the Sobolev constants along the Calabi flow geometrically. Combining all these together, we can prove that the Calabi flow exists for all time and converges to an extK metric in many classes on toric surfaces with $c_1>0$.  For simplicity,  we only present a particular example of our results.

We consider $M\sim\mathbb{CP}^2 \sharp 3 \overline{\mathbb{CP}^2}$. $M$ can be obtained by $\mathbb{CP}^2$ blown up at three generic points. This is known as a del Pezzo surface with tori symmetry. Its automorphism group contains a compact two-torus $T^2$. A metric which is invariant under the action of $T^2$ is called a toric metric. Calabi showed that any extK metric admits the maximal symmetry that the manifold allows \cite{Calabi85}; namely an extK metric is invariant under the identity component of a maximal compact subgroup of the automorphism group of the manifold. 
This implies that any extK metric on $M$ has to be toric. Let $H$ be a hyperplane in $\mathbb{CP}^2$. After blown up, we still use $H$ to denote the corresponding hypersurface on $M$ and $E_i, i=1, 2, 3$ to denote the exceptional divisors.  For simplicity, we use $[H], [E_i]$ to denote the homology classes and their Poincar\'e dual-the cohomology classes.
A K\"ahler class on $M$ can be described as $[\omega]_{x, y, z}=3[H]-(x[E_1]+y[E_2]+z[E_3])$. When $x=y=z$,
the K\"ahler classes $[\omega]_x=3[H]-x([E_1]+[E_2]+[E_3])$ are invariant under the action of $\mathbb{Z}_3$, which acts on exceptional divisors as a cyclic group. Note that the positivity condition of a K\"ahler class requires that $x\in (0, 3/2)$ and  $M$ degenerates to $\mathbb{CP}^2$ when $x=0, 3/2$.  We will give a full detailed discussion for $x=1/2$. We can also prove similar results for any $x\in (0, 3/2)$ using the same method.  We will give a rough discussion and state our theorem for $x\in (0, 3/2)$ in appendix.

\begin{theo}\label{T-2}Let $[\omega]=3H-1/2([E_1]+[E_2]+[E_3])$ on $M$. Suppose the initial metric $\omega_0\in [\omega]$ is invariant under the toric action and the action of $\mathbb{Z}_3$ and
\[
\int_MR^2dg_0<32\pi^2(6+25/11)
\] 
then the Calabi flow exists for all time with uniformly bounded curvature tensor, which converges to a cscK metric in the class $[\omega]$ in the Cheeger-Gromov sense. 
\end{theo}

\begin{rem}A priori lower bound of the Calabi energy is $32\pi^2 (75/11)$ for $x=1/2$. Our result  actually implies the existence of a cscK metric in this class which realizes this lower bound. 
\end{rem}
The toric structure is only used  to classify {\it maximal bubbles}. It will be interesting to remove this assumption and that will lead to some general existence of extK metrics on K\"ahler surfaces with $c_1>0$. 

The existence of extK metric is one of the central problems in
K\"ahler geometry. One can see  rapid developments in the last
few years. We want to refer readers to
\cite{apocalgau01, Dona04, Dona03, chen-tian, AP} etc.
for further references in this subject.

The organization of the paper goes as follows: in Section 2 we
recall some basic facts in K\"ahler geometry and  the Calabi
flow. In Section 3 we calculate the evolution equations for
the curvature tensor and derive some integral estimates along the Calabi flow. In Section 4  we discuss the formation of
the singularities. We will prove Theorem \ref{1} and use it to prove some general properties. In Section 5 we give some detailed description of scalar flat K\"ahler surfaces, in particular with toric assumption. In Section 6 we prove Theorem \ref{T-2}. We exhibit the details, for the particular case, how to bound the Sobolev constants and  how to rule out bubbles. In appendix we discuss our results roughly in general case.

\vspace{2mm}

\noindent {\bf Acknowledgement}: The first named author wishes to
thank G. Tian for insightful discussions on this topic. He is also grateful to S. K.
Donaldson and C. LeBrun for many discussions on four dimensional
geometry.  The second named author would like to thank Song Sun and Brian Weber for insightful discussions. He is also grateful to Zheng Hua, Jeff Viaclovsky and Bing Wang for valuable
discussions. The present version is thoroughly rewritten and the authors are grateful to the referee for numerous suggestions. 

\section{Preliminary}
Let $(M, [\omega], J)$ be a compact K\"ahler manifold
of complex dimension $n$. A K\"ahler form  on $M$ in local
coordinates is given by \begin{eqnarray*} \omega = \sqrt{-1}g_{i\bar{j}}dz^i\wedge dz^{\bar{j}},\end{eqnarray*} where
$\{g_{i\bar{j}}\}$ is a positive definite Hermitian matrix.  The K\"ahler condition is $d\omega=0$. In local
coordinates, it says that
\[
{{\p g_{i \bar j}\over {\p z_{\bar k}}}} = {{\p g_{i \bar k}\over
{\p z_{\bar j}}}}, \qquad {{\p g_{i \bar j}\over {\p z_{ k}}}} =
{{\p g_{k \bar j}\over {\p z_{i}}}}.
\]
 Note that the K\"ahler class $[\omega] \in H^2(M, \mathbb{R})$ is a nontrivial cohomology class.
By $\p\bar\p$-lemma, any other K\"ahler form in the same cohomology class is
of the form
\begin{eqnarray*}\omega_{\varphi}=\omega+\i
\p\bar{\p}\varphi>0,\end{eqnarray*} for some real valued function
$\varphi$ on $M.\;$
The corresponding K\"ahler metric is
denoted by
$g_\varphi=\left(g_{i\bar{j}}+\varphi,_{i\bar{j}}\right)dz^i\otimes
dz^{\bar{j}}.\;$  For
simplicity, we use both $g$ and $\omega$ to denote the K\"ahler
metric. Define the space of K\"ahler potentials  as \begin{eqnarray*}
\cH_\omega=\left\{\varphi|\omega_\varphi=\omega+\sqrt{-1}\partial\bar{\partial}\varphi>0, ~~
\varphi \in C^{\infty}(M) \right\}.\end{eqnarray*}

Given a K\"ahler metric $\omega$, its volume form is \[
dg:=
\frac{\omega^n}{n!}=(\sqrt{-1})^n\det{(g_{i\bar{j}})}dz^1\wedge
dz^{\bar{1}}\wedge\cdots \wedge dz^n\wedge dz^{\bar{n}}.\] The
Ricci curvature of $\omega$ is given by \[
R_{i\bar{j}}=-\p_i\p_{\bar{j}}\log {\det{(g_{k\bar{l}})}}.\] Its
Ricci form $\rho$ is of the form \begin{eqnarray*} \rho=\i
R_{i\bar{j}}dz^i\wedge dz^{\bar{j}} =-\i \p_i \p_{\bar
j}\log{\det{(g_{k\bar{l}})}}dz^i\wedge dz^{\bar{j}} .\end{eqnarray*} It is a real, closed
$(1,1)$ form. The cohomology class of  the Ricci form is the famous
first Chern class $c_1(M)$, independent of the metric.

In 1980s,  E. Calabi \cite{Calabi82, Calabi85} introduced
the Calabi functional as  \begin{eqnarray*}
\cC(\omega_\varphi)=\int_MR_{\varphi}^2dg_\varphi,\end{eqnarray*}
where $R_\varphi$ is the scalar curvature of $\omega_\varphi$. Let
$\underline{R}$ be the average of the scalar curvature, which is a
constant depending only on the class $[\omega]$ and the underlying
complex structure. Usually we use the following modified Calabi
energy \begin{eqnarray*} \tilde{\cC}
(\omega_\varphi)=\int_M\left(R_{\varphi}-\underline{R}\right)^2dg_\varphi\end{eqnarray*}
to replace $\cC(\omega_\varphi)$ since they only differ by a
topological constant. Calabi studied the variational problem to
minimize $\int_M|Rm|^2dg$ in $\cH_\omega$, which is equivalent to minimize the Calabi functional $\cC(\omega)$. 
A critical point is either a cscK metric or an extK metric depending on whether the Futaki character \cite{Futaki, Calabi85}
vanishes or not. The Futaki character $\cF=\cF_\varphi:
\h(M)\rightarrow \mathbb{C}$ is defined on the Lie algebra $\h(M)$
of all holomorphic vector fields of $M$ as follows,
\begin{eqnarray*}
\cF_\varphi(X)=-\int_MX(F_\varphi)dg_\varphi,\end{eqnarray*}
where $X \in \h(M)$ and $F_\varphi$ is a real valued function
defined by
\begin{eqnarray*} \triangle_\varphi\;  F_\varphi=R_\varphi - \underline{R},\end{eqnarray*}
where $\triangle_\varphi$ is the Laplacian operator of the metric
$\omega_\varphi.$
The Calabi flow \cite{Calabi82} is defined by
\begin{eqnarray*}
\frac{\p \varphi}{\p t}=R_\varphi-\underline{R}.\end{eqnarray*}
Under the Calabi flow, we
have\begin{eqnarray*}\frac{d}{dt}\int_M(R_\varphi-\underline{R})^2dg_\varphi=-2\int_M\left(D_\varphi
R_\varphi, R_\varphi\right) dg_\varphi  \leq
0,\end{eqnarray*} where $D_\varphi$ is Lichn\'erowicz operator
with respect to $\omega_\varphi.$ The Lichn\'erowicz
operator  $D$ is defined by
\begin{eqnarray*} Df={f,_{\alpha\beta}}^{\alpha\beta},\end{eqnarray*} where
the covariant derivative is with respect to $\omega$.
The Calabi flow enjoys many wonderful properties.  For instance,
it is the gradient flow of both the Calabi energy and the Mabuchi energy. It also
decreases the distance function in the space of K\"ahler metrics \cite{Calabi-Chen}.
Assuming the long time existence, S. Donaldson \cite{Dona02}
described the convergence properties of the Calabi flow
conceptually and made conjectures regarding the convergence of the
Calabi flow, and he also related the convergent behavior of the
Calabi flow to the stability conjecture (Yau-Tian-Donaldson
conjecture) in K\"ahler geometry. 

\section{The evolution equations along the Calabi flow}
In this section, we study the evolution equations of the curvature
tensor under the Calabi flow and prove some integral estimates. In
particular, we give  interpolation formulas of  curvature tensors
under the Calabi flow, similar to the results of R. Hamilton in
Ricci flow \cite{Hamilton82}. 
\begin{theo}\label{T-3-1}
If the
curvature is bounded along the Calabi flow, then
\[
\frac{\p}{\p
t}\int_M|\nabla^kRm|^2dg\leq-\frac{1}{2}\int_M|\nabla^{k+2}Rm|^2dg+C\int_M|Rm|^2dg,
\]
where $C=C(n, k, \max|Rm|)$ is a constant.
\end{theo}

The evolution equation of the
curvature tensor (Riemannian curvature) is given by
\begin{equation}\label{3-1}
\begin{split}
\frac{\p}{\p t}Rm&=-\nabla
\bar{\nabla}\nabla\bar{\nabla}R+Rm*\nabla^2R\\&=-\triangle^2Rm+\nabla^2Rm*Rm+\nabla
Rm*\nabla Rm,
\end{split}
\end{equation} where we use $*$ to denote possible contractions of tensors for simplicity. 
\begin{lemma}
The evolution for high derivatives of Riemannian curvature is
\begin{equation}\label{3-2}
\frac{\p}{\p t}\nabla^kRm=-\triangle^2\nabla
^kRm+\sum_{i+j=k+2}\nabla^iRm*\nabla^jRm.
\end{equation}
\end{lemma}
\begin{proof}
Note that for any tensor $T$, by interchanging the derivatives
\[
\nabla \triangle T=\triangle \nabla T+\nabla Rm*T+Rm*\nabla T.
\]
If $k=0$, this is true by (\ref{3-1}). We proceed by
induction on $n$. This gives
\begin{eqnarray*}
\frac{\p}{\p t}\nabla^{k+1}Rm&=&\nabla(\frac{\p}{\p
t}\nabla^kRm)+\nabla^2Rm*\nabla^{k+1}Rm\\
&=&\nabla\left(-\triangle^2\nabla^kRm+\sum_{i+j=k+2}\nabla^iRm*\nabla^jRm\right)+\nabla^2Rm*\nabla^{k+1}Rm\\
&=&-\nabla\triangle^2\nabla^kRm+\sum_{i+j=k+3}\nabla^iRm*\nabla^jRm\\
&=&-\triangle^2\nabla^{k+1}Rm+\sum_{i+j=k+3}\nabla^iRm*\nabla^jRm.
\end{eqnarray*}
This completes the proof.
\end{proof}
\begin{lemma}
In complex setting, for any tensor $T$, the different norms are
relevant as the following,
\begin{equation}\label{3-3}
\int_M|\triangle T|^2dg=\int_M|\nabla\bar\nabla
T|^2dg+\int_MRm*T*\nabla^2T\end{equation} and
\begin{equation}\label{3-4}
\int_M|\triangle
T|^2dg=\frac{1}{2}\int_M|\nabla^2T|^2dg+\int_MRm*T*\nabla^2
Tdg+\int_MRm*\nabla T*\nabla Tdg.
\end{equation}
\end{lemma}
\begin{proof}
Integration by parts,
\[
\int_M|\triangle
T|^2dg=\int_MT,_{i\bar{i}}T,_{j\bar{j}}dg=\int_MT,_{i\bar{j}}T,_{j\bar{i}}dg=\int_MT,_{i\bar{j}}T,_{\bar{i}j}dg+\int_MRm*T*\nabla^2T.
\]
This gives (\ref{3-3}). Similarly,
\[
\int_M|\triangle T|^2dg=\int_MT,_{ij}T,_{\bar i\bar
j}dg+\int_MRm*T*\nabla Tdg+\int_MRm*\nabla T*\nabla Tdg.
\]
Combine them,
\[
\int_M|\triangle
T|^2dg=\frac{1}{2}\int_M|\nabla^2T|dg+\int_MRm*T*\nabla
Tdg+\int_MRm*\nabla T*\nabla Tdg.
\]
This completes the proof.
\end{proof}
\begin{lemma}
Along the Calabi flow, we have
\begin{equation}\label{3-5}
\begin{split}
\frac{\p}{\p
t}\int_M|\nabla^kRm|^2dg&=-\int_M|\nabla^{k+2}Rm|^2dg
+\sum_{i+j=k+2}\int_M\nabla^iRm*\nabla^jRm*\nabla^kRmdg\\
&\quad+\int_MRm*\nabla
^{k+1}Rm*\nabla^{k+1}Rmdg.
\end{split}
\end{equation}
\end{lemma}
\begin{proof}We compute
\begin{equation*}
\begin{split}
\frac{\p}{\p t}\int_M|\nabla^kRm|^2dg&=\int_M\left\langle
\frac{\p}{\p t}\nabla^kRm, \overline{\nabla^nRm}\right\rangle
dg+\int_M\left\langle \nabla^kRm, \frac{\p}{\p
t}\overline{\nabla^nRm}\right\rangle dg\\
&\quad+\int_M\nabla^2Rm*\nabla^k Rm*\nabla^kRmdg.
\end{split}
\end{equation*}
By (\ref{3-2}), we compute
\[
\frac{\p }{\p
t}\nabla^kRm=-\triangle^2\nabla^kRm+\sum_{i+j=k+2}\nabla^iRm*\nabla^jRm,
\]
\[
\frac{\p}{\p t}\overline{\nabla^kRm}=\overline{\frac{\p}{\p
t}\nabla^kRm}=-\triangle^2\overline{\nabla^kRm}+\sum_{i+j=k+2}\nabla^iRm*\nabla^jRm.
\]
It gives that
\begin{eqnarray*}
\frac{\p}{\p t}\int_M|\nabla^kRm|^2dg&=&-2\int|\triangle
\nabla^kRm|^2dg+\sum_{i+j=k+2}\int_M\nabla^iRm*\nabla^jRm*\nabla^kRmdg\\
&=&-\int_M|\nabla^{k+2}Rm|^2dg+\sum_{i+j=k+2}\int_M\nabla^iRm*\nabla^jRm*\nabla^kRmdg\\
\quad&&+\int_MRm*\nabla^{k+1}Rm*\nabla^{k+1}Rmdg.
\end{eqnarray*}
\end{proof}
To derive integral estimates, we need some interpolation
inequalities for tensors. These inequalities (\ref{3-6})
(\ref{3-7}) and (\ref{3-8}) are due to R. Hamilton \cite{Hamilton82} and we list them below for the sake of completeness. 
Note that the estimates
do not depend on the metrics evolved. 
\begin{lemma}(Interpolation inequalities for tensors)
Let $M$ be a compact Riemannian manifold of dimension $n$ and $T$
be any tensor on $M$. Suppose
\[
\frac{1}{p}+\frac{1}{q}=\frac{1}{r}\quad\mbox{with}\quad r\geq
1.
\]
Then
\begin{equation}\label{3-6}
\left\{\int_M|\nabla T|^{2r}dg\right\}^{\frac{1}{r}}\leq
(2r-2+n)\left\{\int_M|\nabla^2T|^pdg\right\}^{\frac{1}{p}}\left\{\int_M|T|^qdg\right\}^{\frac{1}{q}}.
\end{equation}
\end{lemma}
\begin{cor}
\begin{equation}\label{3-7}
\int_M|\nabla^iT|^{2k/i}dg\leq
C\max_M|T|^{2k/i-2}\int_M|\nabla^kT|^2dg.
\end{equation}
\end{cor}
\begin{cor}
\begin{equation}\label{3-8}
\int_M|\nabla^iT|^2dg\leq
C\left\{\int_M|\nabla^kT|^2dg\right\}^{i/k}\left\{\int_M|T|^2dg\right\}^{1-i/k}.
\end{equation}
\end{cor}

Now we are ready to to prove Theorem \ref{T-3-1}.
\begin{proof}
Recall
\begin{eqnarray*}
\frac{\p}{\p t}\int_M|\nabla^kRm|^2dg&=&-\int_M|\nabla^{k+2}Rm|^2dg+\sum_{i+j=k+2}\int_M\nabla^iRm*\nabla^jRm*\nabla^kRmdg\\
\quad&&+\int_MRm*\nabla^{k+1}Rm*\nabla^{k+1}Rmdg.
\end{eqnarray*}
We know that
\[
\left|\int_MRm*\nabla^{k+1}Rm*\nabla^{k+1}Rmdg\right|\leq
C\max_M|Rm|\int_M|\nabla^{k+1}Rm|^2dg,
\]
and
\begin{eqnarray*}
&&\left|\int_M\nabla^iRm*\nabla^jRm*\nabla^nRmdg\right|\\
\quad &&\leq
\left\{\int_M|\nabla^iRm|^{(2k+4)/i}dg\right\}^{i/(2k+4)}
\left\{\int_M|\nabla^jRm|^{(2k+4)/j}dg\right\}^{j/(2k+4)}\left\{\int_M|\nabla^kRm|^2dg\right\}^{1/2}
\end{eqnarray*}
with $i+j=k+2.$ By the interpolation result (\ref{3-7}), we have
($i\neq 0$)
\[
\left\{\int_M|\nabla^iRm|^{(2k+4)/i}dg\right\}^{i/(2k+4)}\leq
C\max_M|Rm|^{1-i/(k+2)}\left\{\int_M|\nabla^{k+2}Rm|^2dg\right\}^{i/(2k+4)},
\]
and doing the same for $j$, it follows
\[
\left|\int_M\nabla^iRm*\nabla^jRm*\nabla^nRmdg\right|\leq
C\max_M|Rm|\left\{\int_M|\nabla^{k+2}Rm|^2dg\right\}^{1/2}\left\{\int_M|\nabla^kRm|^2dg\right\}^{1/2}.
\]
By the interpolation inequality (\ref{3-8}) and Young's
inequality,
\[
\int_M|\nabla^iRm|^2dg\leq
\epsilon\int_M|\nabla^{k+2}Rm|^2dg+C\int_M|Rm|^2dg
\]
for any $1\leq i<k+2.$ It follows that
\begin{equation}
\begin{split}
\frac{\p}{\p
t}\int_M|\nabla^kRm|^2dg&\leq-\int_M|\nabla^{k+2}Rm|^2dg+C\int_M|\nabla^{k+1}Rm|^2dg\\
&\quad+C\left\{\int_M|\nabla^{k+2}Rm|^2dg\right\}^{1/2}\left\{\int_M|\nabla^kRm|^2dg\right\}^{1/2}\\
&\leq-\int_M|\nabla^{k+2}Rm|^2dg+\epsilon\int_M|\nabla^{k+2}Rm|^2dg+C\int_M|Rm|^2dg\\
&\leq-\frac{1}{2}\int_M|\nabla^{k+2}Rm|^2dg+C\int_M|Rm|^2dg.
\end{split}
\end{equation}
\end{proof}

\section{Formation of singularities along the Calabi flow}
In this section we study the formation of singularities along the Calabi flow with uniformly bounded Sobolev constants.  
In particular we will prove Theorem \ref{1} and use it to prove some general properties.
Note that the Calabi flow is a fourth order parabolic equation, the
maximum principle is not applicable in general.  The tools we  use are
the  integral estimates proved in Section 3 and Sobolev inequalities.  
First we need a Sobolev inequality for tensors. First recall the Sobolev inequality on a compact Riemannian manifold $(M, g)$ of dimension $m$; 
there exists a constant $C_s=C_s(M, g)$ such that
\begin{equation}\label{E-4-1}
\|f\|_{L^{2m/(m-2)}}\leq C_s \left(\|\nabla f\|_{L^2}+Vol^{-1/m}\|f\|_{L^2}\right),
\end{equation}
where $Vol$ denotes the volume of $(M, g)$. Note that the Sobolev inequality \eqref{E-4-1} is scaling invariant. 
\begin{lemma}[Sobolev inequality for tensors]\label{L-6-2} Let $(M, g)$ be a
compact Riemannian manifold of dimension $m$ with fixed volume, and let $T$ be any
tensor on $M$. Then the Sobolev inequality holds for $T$, namely,
\[
\|T\|_{L^{2m/(m-2)}}\leq C_s(\|\nabla T\|_{L^2}+Vol^{-1/m}\|T\|_{L^2}).
\]
\begin{proof}
Denote $f=|T|$, the Sobolev inequality for functions gives that
\[
\|f\|_{L^{2m/m-2}}\leq C_s(\|\nabla f\|_{L^2}+Vol^{-1/m}\|f\|_{L^2}).
\]
Since $|\nabla f|=|\nabla |T||=|\nabla \langle T, T\rangle^{1/2}|,$
and when $|T|\neq 0$,
\[
|\nabla \langle T, T\rangle^{1/2}|=\frac{\langle \nabla T,
T\rangle}{|T|}\leq |\nabla T|.
\]
If $|T|0$ vanishes at some point, take $f=\sqrt{\langle T, T\rangle+\epsilon}$ and
apply the Sobolev inequality for $f$, then let $\epsilon
\rightarrow 0$. This completes the proof.
\end{proof}
\end{lemma}

Now by the integral estimates derived in Section 3, we can prove
\begin{lemma}\label{L-6-3}
If the curvature tensors and the Sobolev constants evolved are both uniformly bounded in $[0, T)$
($T\leq \infty$ ), then all derivatives of curvature are
bounded (independent of time $T$).
\end{lemma}
\begin{proof}
Recall that
\[
\frac{\p}{\p
t}\int_M|\nabla^kRm|^2dg\leq-\frac{1}{2}\int_M|\nabla^{k+2}Rm|^2dg+C\int_M|Rm|^2dg,
\]
if the curvature tensor is uniformly bounded by Theorem \ref{T-3-1}. For
any $t_0 \in [0, T-1)$  denote
\begin{eqnarray*}
F_k(t)=\sum^k_{i=0}t^i\int_M|\nabla^iRm|^2dg(t_0+t),
\end{eqnarray*}
where $t\in [0, 1]$. Then we will have
\begin{eqnarray*}
\frac{\p F_k}{\p
t}&=&\sum^k_{i=0}it^{i-1}\int_M|\nabla^iRm|^2dg+\sum^k_{i=1}t^i\frac{\p}{\p
t}\int_M|\nabla^iRm|^2dg\\
&\leq&\sum^k_{i=0}it^{i-1}\int_M|\nabla^iRm|^2dg+\sum^k_{i=1}t^i\left(-\frac{1}{2}\int_M|\nabla^{i+2}Rm|^2dg+C\int_M|Rm|^2dg\right)\\
&\leq&\sum^k_{i=0}t^{i}\left(-\frac{1}{2}\int_M|\nabla^{i+2}Rm|^2dg+(i+1)\int_M|\nabla^{i+1}Rm|^2dg\right)+C\int_M|Rm|^2dg\\
&\leq&\sum^k_{i=0}t^{i}\left(-\frac{1}{4}\int_M|\nabla^{i+2}Rm|^2dg+C\int_M|Rm|^2dg\right)+C\int_M|Rm|^2dg\\
&\leq&-\frac{1}{4}\sum^k_{i=0}t^{i}\int_M|\nabla^{i+2}Rm|^2dg+C\int_M|Rm|^2dg.
\end{eqnarray*}
It follows that
\begin{eqnarray*}
F_k(1)&=&F_k(0)+\int_0^1\frac{\p F_k}{\p t}dt\\
&\leq& \int_M|Rm|^2dg(t_0)+\int_0^1C\int_M|Rm|^2dg\\
&\leq&C\int_M|Rm|^2dg(t_0).
\end{eqnarray*}
Since $F_k(1)=\sum^k_{i=0}\int_M|\nabla^iRm|^2dg(t_0+1)$, it means
$L^2$ norms of all higher derivatives of curvature are bounded.
With uniformly bounded Sobolev constants,  it implies that all
higher derivatives of the curvature are uniformly bounded
independent of time by Sobolev embedding theorem.
\end{proof}
Now, we are ready to prove Theorem \ref{1}
\begin{proof}
Consider the Calabi flow in $[0, T)$ with uniformly bounded Sobolev constants
($T \leq\infty$). If the curvature blows up when
$t\rightarrow T,$ pick up a sequence $t_i\rightarrow T$, denote
\[Q_i=\max_{x\in M, t\leq t_i}|Rm|(x, t_i)=|Rm|(x_i, t_i).\]
Re-scale the metric
\[
g_i(t)=Q_ig(t/Q_i^2+t_i).
\]
Note that the Sobolev inequality is scaling invariant.  After
scaling, the curvature tensors are uniformly bounded. It follows from Lemma
\ref{L-6-3} that all the higher derivatives of the curvature are
uniformly bounded for any $g_i(t)$. It is also clear that 
bounded Sobolev constants with bounded curvature tensors  imply that
the injectivity radius is bounded away from zero at any point. One
can find a nice reference in \cite{Ye} for all details.  Then the pointed manifold $\{M, g_i(t), x_i\}$ converges subsequently
smoothly to $\{M_\infty, g_\infty(t), x_\infty\}$ in
Cheeger-Gromov sense. 
The proof is exactly the same as in the
Ricci flow shown by R. Hamilton \cite{Hamilton951}. 
In particular $\{g_\infty(t), -\infty<t\leq0\}$ is an ancient solution of the
Calabi flow
\[
\frac{\p \tilde g}{\p t}=\nabla\bar{\nabla}\tilde R
\]
on the complete manifold (non-compact) $M_\infty$.
Now we consider the convergence of the complex structure. By the blowing up process, there is a sequence of
compact set $K_i$, $K_i\subset K_{i+1}$, $\cup K_i=M_\infty$, and
a sequence of  diffeomorphisms $\Phi_i: K_i\rightarrow
\Phi_i(K_i)\subset M$,  \[\Phi_i^{*}(g_i)\rightarrow g_\infty,\] where the convergence is  smooth in $K_{i-1}$. 
Denote $J_i={\Phi^{-1}_i}_*\circ J\circ {\Phi_i}_*$. If necessary, by taking a subsequence, $J_i\rightarrow J_\infty$. It is clear that $J_\infty$ is still a complex structure which is compatible with $g_\infty(t)$ for any $t$.

When $(M, J)$ is a K\"ahler surface, we
can prove that $\{M_\infty, g_\infty(t)\}$ is actually extK 
since the Calabi energy is scaling invariant on K\"ahler
surface. For any fixed time $t\in (-\infty, 0)$, consider
$g_\infty(t)$ as the limit of $g_i(t/Q_i^2+t_i)$ when
$i\rightarrow \infty.$ We have
\begin{eqnarray*}
\int_t^0\int_{M_\infty}|R,_{\alpha\beta}|^2dg_\infty
dt&\leq&\lim_{i\rightarrow \infty}
\int_{t/Q_i^2+t_i}^{t_i}\int_M|R,_{\alpha\beta}|^2dg_i(s)ds\\
&=&\lim_{i\rightarrow
\infty}\{\cC(t/Q_i^2+t_i)-\cC(t_i)\}\\
&=&0.
\end{eqnarray*}
It gives  that for any  $s\in[t, 0]$,
\[
\int_{M_\infty}|R,_{\alpha\beta}|^2dg_\infty=0.
\]
It follows that $g_\infty(s)$ is an extK metric for any $s$, and so $(M_\infty, g_\infty(s))$ are all biholomorphic. Let $g_\infty=g_\infty(0)$.  
In particular on $(M_\infty, g_\infty)$, for any smooth  function $f$ with compact support,  the following Sobolev inequality holds for some constant $C_s=C_s(M_\infty, g_\infty)$,
\begin{equation}\label{E-S} 
\|f\|_{L^{n/n-1}}\leq C_s\|\nabla f\|_{L^2}.
\end{equation}
It is also clear that  the $L^2$ norm of curvature on $(M_\infty, g_\infty)$ is
bounded. By Chen-Weber's work on the moduli of extremal
metrics \cite{chen-weber} on K\"ahler surfaces,  $(M_\infty,
g_\infty)$  has ALE (asymptotically locally Euclidean) structure at infinity. We then finish the proof by showing the following lemma.
\end{proof}

\begin{lemma}\label{T-4-3}Let $(M_\infty, g_\infty, J_\infty)$ be an ALE extK metric with complex dimension $m>1$, then $(M_\infty,
g_\infty, J_\infty)$ is scalar flat.
\end{lemma}

ALE spaces have been studied often in the
literature. A complete manifold $(X, g)$ is called
asymptotically locally Euclidean (ALE) end of order $\tau (\tau>0)$ if
there exists a finite subgroup $\Gamma\subset SO(n)$ acting freely
on $\mathbb{R}^n\backslash B(0, R)$ and a $C^{\infty}$
diffeomorphism $\Psi: X\rightarrow (\mathbb{R}^n\backslash B(0,
R))/\Gamma$ such that under this identification,
\begin{eqnarray*}
 g_{ij}&=&\delta_{ij}+O(r^{-\tau}),\\
\p^{|k|}g_{ij}&=&O(r^{-\tau-k}),
\end{eqnarray*}
for any partial derivative of order $k$ as $r\rightarrow \infty.$
We say an end is ALE is of order $0$ if we can find a coordinate
system as above with
\begin{eqnarray*}
 g_{ij}&=&\delta_{ij}+o(1),\\
\p^{|k|}g_{ij}&=&o(r^{-k}),
\end{eqnarray*}
as $r\rightarrow \infty.$ A complete, non-compact manifold $(X,
g)$ is called ALE is X can be written as the disjoint union of a
compact set and finitely many ALE ends. When $(X, g)$ is in
addition a K\"ahler manifold, $\Gamma\subset Gl(n/2, \mathbb{C})\cap SO(n)=U(n/2)$, $n=2m.$ In
particular one can show that an ALE K\"ahler manifold has just one
end, as results of Li and Tam \cite{Li-Tam}.\\

Now we are ready to prove Lemma \ref{T-4-3}.

\begin{proof} Let  $(M_\infty, g_\infty)$ be an ALE K\"ahler manifold with an extK metric $g_\infty$. Since
$\Gamma\subset U(m)$,  there exists an asymptotic coordinate
$\mathbb{C}^m\backslash B(0, R)/\Gamma$ on the ALE end. Through
$\pi: \mathbb{C}^{m}\backslash B(0, R)\rightarrow
\mathbb{C}^{m}\backslash B(0, R)/\Gamma$, the metric can be lifted
up to $\mathbb{C}^{m}\backslash B(0, R).$ If $g_\infty$ is extK, then $\nabla R$ is the
real part of a holomorphic vector, on $\mathbb{C}^m\backslash B(0,
R)$, and one can express $\nabla R$ as
\[
\sum_{i=1}^mf_i(z_1, \cdots, z_m)\frac{\p}{\p
z_i}+\overline{f_i(z_1, \cdots, z_m)}\frac{\p}{\p\bar z_i},
\]
where $f_i~(1\leq i\leq m)$ are holomorphic functions on
$\mathbb{C}^m\backslash B(0, R)$. By Hartgos' theorem, one can
extend $f_i$ to $\mathbb{C}^m$. Since the metric is ALE,  $|f_i|$ vanishes at infinity. Now the real part and imaginary
part of $f_i$ are both harmonic functions, the maximum principle
implies  $f_i\equiv 0$. So $\nabla R\equiv 0.$ It follows that
$R=const.$ The ALE condition implies that $R$ has to be zero.
\end{proof}

In the following we shall discuss some general properties along the Calabi flow using Theorem \ref{1}. First we define type I and type II singularities for the Calabi flow.
\begin{defi} Suppose the Calabi flow exists in $[0, T)$ for some finite time $T$ and the curvature blows up at $T$. A singularity is called type I if
\[
\limsup_{t\rightarrow T}
|Rm|^2(T-t)<\infty.
\] 
Otherwise it is called a type II singularity.
\end{defi}
\begin{prop}Let $T$ be the maximal existence time of the Calabi flow on a K\"ahler surface.
If $T<\infty$ and the Sobolev constant is uniformly bounded, then there is no type I singularity. Namely,
\begin{equation*}\label{6-1}\limsup_{t\rightarrow T}
|Rm|^2(T-t)=\infty.\end{equation*}
\end{prop}
\begin{proof}
If not, then
\[\limsup_{t\rightarrow T} |Rm|^2(T-t)=C\]
for some constant $C$. We can choose a sequence $(x_i, t_i)$ such
that \[Q_i:=|Rm|(x_i, t_i)=\max_{0\leq s\leq t_i, y\in M}|Rm|(y, s)\]
and \[\lim_{i\rightarrow \infty} Q_i^2(T-t_i)=C.\]  Then
 \[\left\{M,
g_i(s)=Q_ig(t_i+sQ_i^{-2})\right\}
\]
converges to one parameter family of {\it maximal bubble}
$(M_\infty, g_\infty(s))$ evolved by the Calabi flow. By Theorem \ref{1}, $g_\infty(s)$ is scalar flat and  $g_\infty(s)\equiv g_\infty(0)$ for any $s$.
In particular,  for any $s\in (-\infty, 0]$,
\[\lim_{i\rightarrow \infty}Q_i^{-1}|Rm(x_i, t_i+sQ_i^{-2})|=1.\]
At time $t_i-Q_i^{-2} (s=-1)$, we have  
\[
\lim_{i\rightarrow \infty}|Rm(x_i, t_i-Q_i^{-2})|^2(T-(t_i-Q_i^{-2}))=\lim_{i\rightarrow \infty}Q_i^2(T-t_i)+1=C+1.
\]
It contradicts \[\limsup_{t\rightarrow T} |Rm|^2(T-t)=C.\]
It completes the proof.
\end{proof}

Theorem \ref{1} gives strong restrictions of possible singularities along the Calabi flow on a K\"ahler surface with  the uniformly bounded Sobolev constants. We actually expect that the longtime existence holds for the Calabi flow on
K\"ahler surfaces in this case. In this direction we can prove that

\begin{theo}\label{T-6-7} Consider the Calabi flow exists in $[0, T)$ on
K\"ahler surface with bounded Sobolev constants, and in addition if
we have
\[
\int_M|Rm|^pdg<\infty
\]
for any $p>2$, then $T=\infty$.
\end{theo}
To prove this theorem, first we derive an estimate on
$W^{1, 2}$ norm of the curvature.
\begin{lemma}\label{L-6-8}The Calabi flow exists in $[0, T)$ with uniformly bounded Sobolev
constant. We have
\begin{equation*}
\frac{\p}{\p t}\int_M|\nabla
Rm|^2dg\leq-\frac{1}{2}\int_M|\nabla^3Rm|^2dg+C(p,
C_s)|Rm|^{3p/(p-2)}_{L^p}+C(p, C_s)|Rm|^3_{L^p}.
\end{equation*}
for any $p>2.$
\end{lemma}
\begin{proof}
Recall the evolution equation
\begin{equation*}
\begin{split}
\frac{\p}{\p t}\int_M|\nabla
Rm|^2dg=&-\int_M|\nabla^3Rm|^2dg+\int_MRm*\nabla^2Rm*\nabla^2Rmdg\\&+\int_MRm*\nabla
Rm*\nabla^3Rmdg\\
=&-\int_M|\nabla^3Rm|^2dg+\int_MRm*\nabla^2Rm*\nabla^2Rmdg\\&+\int_M\nabla
Rm*\nabla Rm*\nabla^2Rmdg.
\end{split}
\end{equation*}
By H\"older inequality we know
\begin{equation*}
\begin{split}
\left|\int_MRm*\nabla^2Rm*\nabla^2Rmdg\right|
&\leq
\int_M|Rm||\nabla^2Rm|^2dg\\
&\leq \left\{\int_M|Rm|^pdg\right\}^{1/p}\left\{\int_M|\nabla
^2Rm|^{2q}dg\right\}^{1/q},
\end{split}
\end{equation*}
where $1/p+1/q=1.$
Using Lemma \ref{L-6-2} (Sobolev inequality for tensors), we obtain
\[
|\nabla^2Rm|_{L^4}\leq C_s(|\nabla^3Rm|_{L^2}+|\nabla^2Rm|_{L^2}).
\]
Combining the interpolation inequality \eqref{3-6} in Section 3, we can get
\begin{equation*}
\begin{split}
\left|\int_MRm*\nabla^2Rm*\nabla^2Rmdg\right|
&\leq
\int_M|Rm||\nabla^2Rm|^2dg\\
&\leq
\left\{\int_M|Rm|^pdg\right\}^{1/p}\left\{\int_M|\nabla
^2Rm|^{2q}dg\right\}^{1/q}\\
&\leq
|Rm|_{L^p}|\nabla^2Rm|^{2a}_{L^4}|\nabla^2Rm|^{2(1-a)}_{L^2}\\
&\leq
C|Rm|_{L^p}C_s(|\nabla^3Rm|_{L^2}+|\nabla^2Rm|_{L^2})^{2a}|\nabla^3Rm|^{4(1-a)/3}_{L^2}\\
&\leq
C|Rm|_{L^p}|\nabla^3Rm|^{2(2+a)/3}_{L^2}+C|Rm|_{L^p}|\nabla^3Rm|^{4/3}_{L^2},
\end{split}
\end{equation*}
where $1/2q=a/4+(1-a)/2,$ and so $a=2-2/q.$ It is easy to see
\[
C|Rm|_{L^p}|\nabla^3Rm|^{2(2+a)/3}_{L^2}\leq
\frac{1}{8}|\nabla^3Rm|^2_{L^2}+C|Rm|^{3q/(2-q)}_{L^p},
\]
and
\[
C|Rm|_{L^p}|\nabla^3Rm|^{4/3}_{L^2}\leq
\frac{1}{8}|\nabla^3Rm|^2_{L^2}+C|Rm|^3_{L^p}.
\]
It gives us that
\[
\left|\int_M\nabla Rm*\nabla Rm*\nabla^2Rmdg\right|\leq
\frac{1}{4}|\nabla^3Rm|^2_{L^2}+C|Rm|^{3p/(p-2)}_{L^p}+C|Rm|^3_{L^p}.
\]
We can also estimate
\begin{equation*}
\begin{split}
\left|\int_M\nabla Rm*\nabla
Rm*\nabla^2Rmdg\right|&\leq\int_M|\nabla^2Rm|\nabla
Rm|^2dg\\
&\leq\left\{\int_M|\nabla^2Rm|^qdg\right\}^{1/q}\left\{\int_M|\nabla
Rm|^{2r}dg\right\}^{1/r}.
\end{split}
\end{equation*}
Recall the interpolation inequality \eqref{3-6},
\[
\left\{\int_M|\nabla
Rm|^{2r}dg\right\}^{1/r}\leq\left\{\int_M|\nabla^2
Rm|^{s}dg\right\}^{1/s}\left\{\int_M| Rm|^{t}dg\right\}^{1/t},
\]
where $1/r=1/s+1/t.$ Choose $s=q, t=p$, it follows that
\[
\left|\int_M\nabla Rm*\nabla
Rm*\nabla^2Rmdg\right|\leq\left\{\int_M|\nabla^2Rm|^qdg\right\}^{2/q}\left\{\int_M|
Rm|^{p}dg\right\}^{1/p},
\]
where $2/q+1/p=1.$  Similarly we can estimate that
\[
\left|\int_M\nabla Rm*\nabla
Rm*\nabla^2Rmdg\right|\leq  \frac{1}{4}|\nabla^3Rm|^2_{L^2}+C|Rm|^{3p/(p-2)}_{L^p}+C|Rm|^3_{L^p}.
\]
Combining these estimates, 
it follows that
\[
\frac{\p}{\p t}\int_M|\nabla
Rm|^2dg\leq-\frac{1}{2}\int_M|\nabla^3Rm|^2dg+C(p,
C_s)|Rm|^{3p/(p-2)}_{L^p}+C(p, C_s)|Rm|^3_{L^p}.
\]
for any $p>2.$
\end{proof}
Now we are in the position to prove Theorem \ref{T-6-7}.
\begin{proof}
If $T<\infty$, Lemma \ref{L-6-8} gives that  $W^{1, 2}$ norm of the curvature is bounded by $L^p$ norm of curvature for any $p>2$.
If the curvature blows up, we can get a sequence $(x_i, t_i)$ such
that
\[Q_i=|Rm|(x_i,
t_i)=\max_{x\in M, t\leq t_i}|Rm|(x, t).\] Then
we know that $\{M, g_i(s)=Q_ig(t_i+Q^2_is), x_i\}$ converges
smoothly (in Cheeger-Gromov sense) to a {\it maximal bubble}
$\{M_\infty, g_\infty, x_\infty\}$. In particular $g_\infty(s )\equiv g_\infty(0)$ for any $s\in (-\infty, 0)$ and $g_\infty(0)$ is scalar flat. Now
\begin{eqnarray*}
\int_{M_\infty}|\nabla Rm_{\infty}|^2dg_\infty&\leq&
\lim_{i\rightarrow \infty}\int_{M}|\nabla Rm|_{g_i}^2dg_i\\
&\leq& \lim_{i\rightarrow \infty}Q_i^{-2}\int_{M}|\nabla
Rm|^2dg=0.
\end{eqnarray*}
It gives that
\[
\int_{M_\infty}|\nabla Rm_{\infty}|^2dg_\infty=0.
\]
It forces that $(M_\infty, g_\infty)$ is a flat metric. That is a contradiction.
\end{proof}

\section{Scalar flat ALE K\"ahler surface}
Suppose $(M_\infty, g_\infty, J_\infty)$ is a scalar flat ALE K\"ahler surface in Theorem \ref{1}. In this section we derive some properties of $(M_\infty, g_\infty)$, which will be used crucially in proving Theorem \ref{T-2}.  Many results in this section follow from \cite{chen-lebrun-weber}.
The following lemma is due to \cite{chen-lebrun-weber}.
\begin{lemma}\label{L-5-11}
$M_\infty$ is diffeomorphic to an open subset $U$ of $M$. If $b_1(M)=0$,  then $b_1(M_\infty)=b_3(M_\infty)=0$, and
$b_2(M_\infty)\leq b_{-}(M).$
\end{lemma}
\begin{proof} $(M_\infty, g_\infty)$ is
obtained by a pointed Gromov-Hausdorff limit of re-scaled versions
of $M$, and after scaling the curvature is uniformly bounded.
Therefore the convergence is smooth (in Cheeger-Gromov sense) in
any compact subset of $M_\infty$, by applying suitable
diffeomorphisms. Since $M_\infty$ is a scalar flat ALE K\"ahler
manifold with one end, $M_\infty$ is diffeomorphic to the interior
of a compact domain $U\subset M_\infty$ with smooth boundary
$S^3/\Gamma$. And $U$ can then  be mapped diffeomorphically into
$M$, resulting in a decomposition \[M\approx U\cup_{S^3/\Gamma}V
\] where $U$ and $V$ are manifolds with boundary, $M_\infty\approx
Int(U)$, and $\p U=\p V=S^3/\Gamma.$  We have the exact sequence
\[0\rightarrow \Omega^{*}(M)\rightarrow \Omega^{*}(U)\oplus \Omega^{*}(V)\rightarrow \Omega^{*}(S^3/\Gamma)\rightarrow
0.\] If $b_1(M)=0$, the Mayer-Vietories sequence implies that both
$U$ and $V$ have $b_1=b_3=0,$ while
\[H^{2}(M, \mathbb{R})=H^{2}(U, \mathbb{R})\oplus H^{2}(V, \mathbb{R}).\]
Since the analogous statement hold similarly for homology, the
intersection form of $M_\infty=Int(U)$ is just the restriction of
the intersection form of $H^{2}(M)$ to the linear subspace
$H^{2}(U)\subset H^2(M)$. But the intersection form of $M_\infty$
is negative ($W_{+}=0$), hence 
$b_2(M_\infty)\leq b_{-}(M).$
\end{proof}
First of all $(M_\infty, J_\infty)$ cannot have trivial topology.
\begin{lemma}[\cite{chen-lebrun-weber}]\label{L-6-12}
Let $(M_\infty, g_\infty)$ be a simply connected scalar flat ALE K\"ahler
metric, then $b_2(M_\infty)\neq 0.$
\end{lemma}

The following lemma appears also in \cite{chen-lebrun-weber}
implicitly. 
\begin{lemma}[Lagrangian condition] If $M_\infty$ contains some
compact 2-sphere $S$, then M contains a sequence of 2-spheres
$S_i$ corresponding to $S$. Moreover,  the homology class of $S_i$
satisfies that
\[
[\o][S_i]\rightarrow 0,
\]
when $i\rightarrow \infty$, which we call Lagrangian condition for
the sequence of $S_i$.
\end{lemma}
\begin{proof} $(M_\infty, g_\infty)$ is constructed as the
pointed limit of $(M, g_i)$ after blowing up. If $M_\infty$
contains a compact 2-sphere $S$, then $(M, g_i)$ contains a
compact 2-sphere for each $i$, and after blowing up, the image of
$S_i$ converges to $S$. Since the scaling factor goes to infinity
as $i\rightarrow \infty$, the area of $S_i$ with respect to $g_i$
must tend to zero. By Wirtinger's inequality, we end up with
\[
\left|[S_i][\o_i]\right|\leq Area_{g_i}(S_i)\rightarrow 0.
\]
Hence
\[
[\o][S_i]=[\o_i][S_i]\rightarrow 0.
\]
It means $S_i$ behaves like a Lagrangian class with respect to $[\o]$ when
$i$ is big.
\end{proof}

The classification of ALE scalar flat K\"ahler metrics on complete
surfaces might be too ambitious. But we know much better if $(M_\infty, g_\infty)$ admits a toric structure. Following \cite{chen-lebrun-weber} (Proposition 16),  we have
\begin{lemma}\label{L-5-12}Let $(M, [\omega])$ be a compact toric surface. Suppose the Calabi flow with the initial metric which is invariant under toric action, if the curvature tensor is not uniformly bounded, then we can form a maximal bubble $(M_\infty, g_\infty)$. 
Moreover $(M_\infty, g_\infty)$ is toric and $H_2(M_\infty,
\mathbb{Z})$ is generated by holomorphic  $\mathbb{CP}^1s$
embedded in $M_\infty$.
\end{lemma}
We can understand $(M_\infty, g_\infty)$ in more detail  by its toric structure. 
If  $b_2(M_\infty)=1$, $H_2(M_\infty, \mathbb{Z})$ is generated by a holomorphic $\mathbb{CP}^1s$ $E$ with self intersection $(-k)$, then the group $\Gamma=\mathbb{Z}_k$ and
\[
c_1=\frac{k-2}{k}[E].
\] 
Hence
\begin{equation}\label{5-1}
\int_{M_\infty}|Ric_0|^2dg_\infty=8\pi^2\frac{(k-2)^2}{k}.
\end{equation}
If  $b_2(M_\infty)=2$, then there is a basis for $H_2(M_\infty,
\mathbb{Z})$, represented by a pair of totally geodesics and
holomorphic $\mathbb{CP}^1s$ $E_1, E_2$, in which the intersection
form becomes\[\left(
\begin{array}{cl}
-k  \hspace{2mm}&1\\
1   \hspace{2mm}&-l\\
\end{array}
\right)\]for some positive $k\geq 2$ and $l\geq 1$  $ (k\geq l).$
Moreover $\Gamma=\mathbb{Z}_{kl-1}$ and
\[
c_1=\left(1-\frac{l+1}{kl-1}\right)[E_1]+\left(1-\frac{k+1}{kl-1}\right)[E_2].
\]
Hence
\begin{eqnarray}\label{5-2}
\int_{M_\infty}|Ric_0|^2dg_\infty=8\pi^2\left\{k\left(1-\frac{l+1}{kl-1}\right)^2+l\left(1-\frac{k+1}{kl-1}\right)^2\right\}\nonumber\\
-16\pi^2\left(1-\frac{l+1}{kl-1}\right)\left(1-\frac{k+1}{kl-1}\right).
\end{eqnarray}

If $b_2(M_\infty)=3,$ there is a basis for $H_2(M_\infty,
\mathbb{Z})$, represented by a triple of totally geodesic
$\mathbb{CP}^1s$ $E_1, E_2, E_3$, in which the intersection form
is given by
\[\left(
\begin{array}{cl}
-i  \hspace{3mm}&1 \hspace{7mm} 0\\
1   \hspace{4mm}&-j \hspace{5mm}1\\
0   \hspace{5mm}&1 \hspace{5mm} -k\\
\end{array}
\right),\] for some positive $i, j, k.$  Moreover, 
\[
c_1=\left(1-\frac{jk}{ijk-i-k}\right)[E_1]+\left(1-\frac{k+i}{ijk-i-k}\right)[E_2]+\left(1-\frac{ij}{ijk-i-k}\right)[E_3].
\]
Hence 
\begin{eqnarray}\label{5-3}
&&\int_{M_\infty}|Ric_0|^2dg_\infty=\nonumber\\ && 8\pi^2\left\{\left(1-\frac{jk}{ijk-i-k}\right)^2i
+\left(1-\frac{k+i}{ijk-i-k}\right)^2j+\left(1-\frac{ij}{ijk-i-k}\right)^2k\right\}\nonumber\\
&&\quad-16\pi^2\left(1-\frac{k+i}{ijk-i-k}\right)\left(2-\frac{j(i+k)}{ijk-i-k}\right).
\end{eqnarray}

\section{Proof of Theorem \ref{T-2}}
It is clear from last section  that if a  maximal bubble is formed along the Calabi flow, it has to be very restrictive, in particular with toric assumption. One certainly expects that such a bubble cannot be formed along the Calabi flow, in particular for some special cases. This actually leads to the fact that the curvature cannot blow up along the Calabi flow. In the following we prove that the Sobolev constants are uniformly bounded and  the bubble cannot be formed for the example in Theorem \ref{T-2}.
The main tools to rule out bubbles are the Gauss-Bonnet
formula and the signature formula. 
For a compact smooth 4-manifold $(M, g)$, the Gauss-Bonnet formula
is
\begin{equation*}
\frac{1}{8\pi^2}\int_M\left(|W_{+}|^2+|W_{-}|^2+\frac{R^2}{24}-\frac{|
Ric_0|^2}{2}\right)dg=\chi(M),
\end{equation*}
and the signature formula is
\begin{equation*}
\frac{1}{12\pi^2}\int_M(|W_{+}|^2-|W_{-}|^2)dg=\tau(M). \end{equation*} Here
$R$ is the scalar curvature and $Ric_0$ is the traceless Ricci
curvature. If $M$ admits an orientation-compatible almost complex
structure, then
\[
c_1^2=2\chi+3\tau.
\]
The K\"ahler condition implies
\begin{equation*}
\int_M|W_{+}|^2dg=\int_M\frac{R^2}{24}dg.
\end{equation*}
Suppose $M=\mathbb{CP}^2\sharp 3\overline{\mathbb{CP}^2}$, we get
\begin{equation}\label{6-1}
\int_M|Ric_0|^2dg=\frac{1}{4}\int_MR^2dg-8\pi^2c_1^2=\frac{1}{4}\int_MR^2dg-48\pi^2,
\end{equation}
and
\begin{equation}\label{6-2}
\int_M|W_-|^2dg=\frac{1}{24}\int_MR^2dg+24\pi^2.
\end{equation}

If $(M_\infty, g_\infty)$ is
any ALE 4-manifold with finite group $\Gamma\subset SO(4)$ at
infinity, then the Gauss-Bonnet formula becomes
\[
\frac{1}{8\pi^2}\int_{M_\infty}\left(|W_{+}|^2+|W_{-}|^2+\frac{R^2}{24}-\frac{|
Ric_0|^2}{2}\right)dg_\infty=\chi(M_\infty)-\frac{1}{|\Gamma|}
\]
and the signature formula becomes\[
\frac{1}{12\pi^2}\int_{M_\infty}(|W_{+}|^2-|W_{-}|^2)dg_\infty=\tau(M_\infty)+
\eta(S^3/\Gamma),
\]
where $\chi(M_\infty)$ is the Euler characteristic of non-compact
manifold $M_\infty$ and $\eta(S^3/\Gamma)$ is called $\eta$
invariant. When $(M_\infty, g_\infty)$ is scalar flat K\"ahler,
the formulas can be reduced to
\begin{equation}\label{6-3}
\frac{1}{8\pi^2}\int_{M_\infty}\left(|W_{-}|^2-\frac{|Ric_0|^2}{2}\right)d
g_\infty=\chi(M_\infty)-\frac{1}{|\Gamma|}
\end{equation}
and
\begin{equation}\label{6-4}
-\frac{1}{12\pi^2}\int_{M_\infty}|W_{-}|^2d\mu_{g_\infty}=\tau(M_\infty)+\eta(S^3/
\Gamma).
\end{equation}

Note that the $L^2$ norm of curvature is scaling-invariant in 4
dimension. If $(M_\infty, g_\infty)$ is a { maximal bubble} of
$(M, g)$, we have
\begin{equation}\label{6-5}\int_M|Ric_0|^2dg\geq
\int_{M_\infty}|Ric_0|^2dg_\infty,\end{equation} and
\begin{equation}\label{6-6} \int_M|W_{-}|^2dg\geq
\int_{M_\infty}|W_{-}|^2dg_\infty.\end{equation} 

To prove Theorem \ref{T-2}, we first show that the Sobolev constants are uniformly bounded along the Calabi flow.
\begin{lemma}
Let $[\omega]=3H-1/2([E_1]+[E_2]+[E_3])$ on $\mathbb{CP}^2 \sharp 3 \overline{\mathbb{CP}^2}$. Suppose the initial metric $\omega_0\in [\omega]$ satisfies
\[
\int_MR^2dg_0<32\pi^2(6+25/11),
\] 
then the Sobolev constants of $\omega(t)$ are uniformly bounded from above once the Calabi flow exists. 
\end{lemma}
\begin{proof}
The proof is a generalization of Tian's idea for positive constant scalar curvature metrics on K\"ahler surfaces \cite{Tian}, \cite{Tian-Viaclovsky}. Suppose $g$ is a K\"ahler metric, the Yamabe constant for the conformal class $[g]$ takes the formula
\[
Y_{[g]}=\inf_{u>0}\frac{\int_M(6|\nabla u|^2+R_{g}u^2)dg}{(\int_Mu^4dg)^{1/2}}.
\]
So for any $u>0$, 
\begin{equation}\label{6-7}
Y_{[g]}\left(\int_Mu^4dg\right)^{1/2}\leq \int_M(6|\nabla u|^2+R_{g}u^2)dg.
\end{equation}
Note that the inequality (\ref{6-7}) holds also for any $u$ since $|\nabla |u||\leq |\nabla u|$ at $u\neq 0$. 
A calculation \cite{Tian-Viaclovsky, chen-lebrun-weber} shows that
\begin{equation}\label{6-8}
Y_{[g]}^2\geq 96\pi^2c_1^2-2\int_MR^2dg. 
\end{equation}
If  the Calabi energy satisfies
\[
\int_MR^2dg<48\pi^2c_1^2,
\]
one can get a priori positive lower bound for $Y_{[g]}^2$. Once $Y_{[g]}>0$ and $R_{g}$ is bounded a priori, the Sobolev constant is bounded from above by (\ref{6-7})
a priori. These two conditions can be checked directly for extremal metrics (or constant scalar curvature metrics), see examples in \cite{chen-lebrun-weber}. 

By a careful analysis, one can generalize this idea to the Calabi flow without any bound on scalar curvature.
Suppose that $g_0\in [\omega]_{1/2}$ on  $\mathbb{CP}^2 \sharp 3 \overline{\mathbb{CP}^2}$ satisfies 
\begin{equation}\label{6-9}
\int_MR^2dg_0<32\pi^2(6+25/11).
\end{equation}
Let $\underline{R}$ be the average of the scalar curvature and let $V$ be the volume of $M$ for any metric in $(M, [\omega])$, then
\[
\underline{R}^2 V=32\pi^2 (75/11).
\] 
It is direct to check that \eqref{6-9} is equivalent to 
\[
\int_M R^2_0dg_0<32\pi^2c_1^2+\frac{1}{3}\underline{R}^2 V,
\]

which we can rewrite as
\[
\int_M(R_{g_0}-\underline{R})^2dg_0<96\pi^2c_1^2-2\int_MR^2_{g_0}dg_0.
\]
By (\ref{6-8}), we get that
\begin{equation}\label{6-10}
Y_{[g_0]}^2>\int_M(R_{g_0}-\underline{R})^2dg_0.
\end{equation}
First we show that $Y_{[g_0]}$ has to be positive.  Pick up a minimizing sequence $u_i$,  and there exist positive $\epsilon_i\rightarrow 0$ such that
\[
Y_{[g_0]}+\epsilon_i=\frac{\int_M(6|\nabla u_i|^2+R_{g_0}u^2_i)dg_0}{(\int_Mu^4_idg_0)^{1/2}}.
\]
It follows that
\[
(Y_{[g_0]}+\epsilon_i)|u_i|^2_{L^4}=6\int_M|\nabla u_i|^2dg_0+\int_MR_{g_0}u_i^2dg_0.
\]
So we can get that
\begin{equation}\label{6-11}
(Y_{[g_0]}+\epsilon_i)|u_i|^2_{L^4}-\int_M(R_{g_0}-\underline{R})u_i^2dg=6|\nabla u_i|^2_{L^4}+\underline{R}|u_i|^2_{L^2}. 
\end{equation}
If $Y_{[g_0]}$ is negative, by (\ref{6-10}), the left hand side of (\ref{6-11}) is negative for some $i$ sufficient large, while the right hand side is positive. Contradiction.

If $Y_{[g_0]}$ is positive, we can rewrite (\ref{6-7}) as
\[
|u|_{L^4}^2-\frac{1}{Y_{[g_0]}}\int_M(R_{g_0}-\underline{R})u^2dg_0\leq 6 Y_{[g_0]}^{-1}|\nabla u|^2_{L^2}+\underline{R}Y_{[g_0]}^{-1}|u|^2_{L^2}.
\]
By the Cauchy-Schwarz inequality, it follows that
\[
|u|_{L^4}^2\leq \frac{6}{Y_{[g_0]}-|R_{g_0}-\underline{R}|_{L^2}}|\nabla u|^2_{L^2}+\frac{\underline{R}}{Y_{[g_0]}-|R_{g_0}-\underline{R}|_{L^2}}|u|_{L^2}. 
\]
The sobolev constant on a four dimensional Riemannian manifold is defined to be the smallest constant $C_s$  such that
\[
|u|_{L^4}^2\leq C_s(|\nabla u|^2_{L^2}+Vol^{-1/2}|u|^2_{L^2}).
\]
Note that this definition is scaling invariant.  In our case,
we can compute that $\underline{R}\sqrt{Vol}=20\pi\sqrt{6/11}$ for $[\omega]_{1/2}$ and it is scaling invariant. 
So we can get that
\[
C_s\leq 20\pi\sqrt{6/11} (Y_{[g_0]}-|R_{g_0}-\underline{R}|_{L^2})^{-1}.
\]
Since the Calabi energy is decreasing,  the above estimate certainly holds along the Calabi flow. 
\end{proof}

Now we are ready to show the details to rule out bubbles. 
\begin{proof}Suppose a maximal bubble $(M_\infty, g_\infty)$ is formed. First we examine the group action  $\mathbb{Z}_3$. Let $M$ be the blown up of $\mathbb{CP}^2$ at the points $[1: 0: 0]$, $[0: 1:0]$ and 
$[0:0:1]$. Then $\mathbb{Z}_3$ is generated by the permutation 
\[
P([z_0: z_1: z_2])=[z_1: z_2 :z_0].
\] 
Note that $(M_\infty, g_\infty)$ arises by scaling the interior of a small domain $U_j$. Now move this domain by the action of $\mathbb{Z}_3$ to obtain the two domains $U_j^{'}, U_{j}^{''}$. If after scaling, the distance for $U_j$ to $U_{j}^{'}$ (and so the distance $U_{j}^{'}$ to $U_{j}^{''}$) is finite, then Cheeger-Gromov limit includes the limits of $U_{j}^{'}, U_{j}^{''}$. Then $(M_\infty, g_\infty)$ is invariant under the induced $\mathbb{Z}_3$ action. Otherwise, three copies of $(M_\infty, g_\infty)$ will form and therefore we have
\begin{equation}\label{6-12}
3\int_{M_\infty}|W_{-}|^2dg_\infty<\int_M|W_{-}|^2dg.
\end{equation}
and \begin{equation}\label{6-13}
3\int_{M_\infty}|Ric_0|^2dg_\infty<\int_M|Ric_0|^2dg.
\end{equation}
In other words, if (\ref{6-12}) or (\ref{6-13}) is not satisfied, then $(M_\infty, g_\infty)$ is invariant under the induced $\mathbb{Z}_3$ action.  This symmetry will simplify the calculation a lot.

By Lemma \ref{L-5-11},  $b_2(M_\infty)\leq 3$ and $(M_\infty, g_\infty)$ satisfies all constraints considered in Section 5. 
We analyze case by case.

Case 1: $b_2(M_\infty)=1$. Then $(M_\infty, g_\infty)$ contains a holomorphic $\mathbb{CP}^1$ with self-intersection $-k$.
Suppose $(M_\infty, g_\infty)$ is not invariant under the induced action of $\mathbb{Z}_3$. By (\ref{5-1}) and (\ref{6-3}), we get that
\begin{equation}\label{6-14}
\int_{M_\infty}|Ric_0|^2dg_\infty=8\pi^2\frac{(k-2)^2}{k}
\end{equation}
and
\begin{equation}\label{6-15}
\int_{M_\infty}|W_-|^2dg_\infty=4\pi^2\frac{k^2+2}{k}. 
\end{equation}

By (\ref{6-2}), (\ref{6-12}) and (\ref{6-15}), we get that
\begin{equation}\label{6-16}
12\pi^2\frac{k^2+2}{k}<\frac{1}{24}\int_MR^2dg+24\pi^2.
\end{equation} However there is no solution of (\ref{6-16}) if the initial Calabi energy 
\[
\int_MR^2dg<258.9\pi^2.
\]  

It implies that $(M_\infty, g_\infty)$ is invariant under the induced $\mathbb{Z}_3$ action. So $(M_\infty, g_\infty)$ is diffeomorphic to a domain which is invariant under the action $\mathbb{Z}_3$. It follows that $M$ contains a sequence of smoothly embedded 2-sphere $S_i$ of self-intersection $-k$, which is invariant under $\mathbb{Z}_3$.  By (\ref{6-1}), (\ref{6-5}) and (\ref{6-14}), we can get that
\[
8\pi^2\frac{(k-2)^2}{k}<\frac{1}{4}\int_MR^2dg-48\pi^2.
\]
It follows that $1\leq k\leq 5$.
Express the homology class $[S_i]$ as
\[
[S_i]=mH-n(E_1+E_2+E_3). 
\]
The Lagrangian condition gives that
\begin{equation}\label{6-17}
3m-3n/2\rightarrow 0,
\end{equation}
while the self-intersection condition gives that
\[
m^2-3n^2=-k.
\]
Suppose that $m, n$ are nonnegative, then we have $m<n$. It follows that $2n^2<k\leq 5$. We get that $n=1, m=0$.  Contradiction with (\ref{6-17}). \\

Case 2: $b_2(M_\infty)=2$.  If $|\Gamma|=1$, $H_2(M_\infty, \mathbb{Z})$ has intersection form
\[\left(
\begin{array}{cl}
-2  \hspace{2mm}&1\\
1   \hspace{2mm}&-1\\
\end{array}
\right).\]
By (\ref{5-2}) (with $k=2, l=1$) and (\ref{6-3}),
we compute
\[
\int_{M_\infty}|Ric_0|^2dg_\infty=16\pi^2,
\]
and 
\[
\int_{M_\infty}|W_-|^2dg_\infty=24\pi^2.
\]
If $|\Gamma|\geq 2$, by (\ref{6-3}), we compute
\[
\int_{M_\infty}|W_-|^2dg_\infty\geq 20\pi^2. 
\]
In either case, (\ref{6-12}) does not hold and so $(M_\infty, g_\infty)$ is invariant under the induced $\mathbb{Z}_3$ action. 
So $M$ contains a sequence of two embedded 2 spheres $S_i, \tilde S_i$ which are both invariant under the action of $\mathbb{Z}_3$ with the intersection form
\[\left(
\begin{array}{cl}
-k  \hspace{2mm}&1\\
1   \hspace{2mm}&-l\\
\end{array}
\right).\]
We can express the homology classes $[S_i], [\tilde S_i]$ as
\[
[S_i]=m_1H-n_1(E_1+E_2+E_3),
\]
and 
\[
[\tilde S_i]=m_2H-n_2(E_1+E_2+E_3).
\]
The Lagrangian condition gives that (for $S_i$)
\begin{equation}\label{6-18}
3m_1-3n_1/2\rightarrow 0.
\end{equation}

The self-intersection condition gives that (for $S_i$)
\begin{equation}\label{6-19} m_1^2-3n_1^2=-k.\end{equation}
Note that if $(m_1, n_1)$ is a solution, $(-m_1, -n_1)$ is also a solution. So we can assume that $m_1>0$ since $m_1$ cannot be zero.
 By (\ref{6-18}), $n_1\geq m_1+1$. By (\ref{6-19}) we get that $k=3n_1^2-m_1^2\geq 3\times 2^2-1=11$. Similarly $l\geq 11$.

However, by (\ref{6-5}) and (\ref{5-2}), we get that
\begin{equation}\label{6-20}
k\left(1-\frac{l+1}{kl-1}\right)^2+l\left(1-\frac{k+1}{kl-1}\right)^2-2\left(1-\frac{l+1}{kl-1}\right)\left(1-\frac{k+1}{kl-1}\right)<2.1.
\end{equation}
For $k, l\geq 11$, (\ref{6-20}) cannot hold.\\

Case 3: $b_2(M_\infty)=3$.  By (\ref{6-3}), we get that
\[
\int_{M_\infty}|W_-|^2dg_\infty\geq 24\pi^2.
\]
By the same argument in Case 2, $(M_\infty, g_\infty)$ is invariant under the induced action of $\mathbb{Z}_3$. It implies that $M$ contains a sequence of three embedded 2-sphere $S_{1a}, S_{2a}, S_{3a} $ which are invariant under the action of $\mathbb{Z}_3$. Suppose the intersection form is given by
\[\left(
\begin{array}{cl}
-i  \hspace{3mm}&1 \hspace{7mm} 0\\
1   \hspace{4mm}&-j \hspace{5mm}1\\
0   \hspace{5mm}&1 \hspace{5mm} -k\\
\end{array}
\right),\] 
then the same argument in Case 2 implies that $i, j, k\geq 11$. However by (\ref{6-5}) and (\ref{5-3}), 
\begin{eqnarray}\label{6-21}\left(1-\frac{jk}{ijk-i-k}\right)^2i
+\left(1-\frac{k+i}{ijk-i-k}\right)^2j+\left(1-\frac{ij}{ijk-i-k}\right)^2k\nonumber\\
-2\left(1-\frac{k+i}{ijk-i-k}\right)\left(2-\frac{j(i+k)}{ijk-i-k}\right)<2.1.
\end{eqnarray}
For $i, j, k\geq 11$, (\ref{6-21}) cannot hold.\\

By the case by case analysis, $(M_\infty, g_\infty)$ cannot be formed. Then the curvature cannot blow up along the Calabi flow. It follows that the flow exists for all time and converges to an extremal metric in Cheeger-Gromov sense. Namely
for any $t_i\rightarrow \infty$, there is a sequence of
diffeomorphism $\Psi_i$ such that $(M, \Psi_i^{*}g_i=g(t_i),
{\Psi_i}_{*}J)$ converge smoothly to $(M, g_\infty, J_\infty)$,
where $g_\infty$ is extremal with respect to $J_\infty$. But $(M,
J_\infty)$ might not be bi-holomorphic to $(M, J)$.

We then finish our proof by showing $(M, J)$ is actually biholomorphic to $(M, J_\infty).$ In general this is a very hard problem. The proof follows from
\cite{chen-lebrun-weber} (Theorem 27) by using the toric  condition carefully and the
classification of complex surface. Theorem 27 in \cite{chen-lebrun-weber} states only for $M\sim\mathbb{CP}^2\sharp 2\overline{\mathbb{CP}^2}$ but the proof holds for all toric Fano surfaces. The key is that in the limiting process, the torus action converges and 
 $(M, g_\infty, J_\infty)$ is still toric. Moreover, the 2-torus action for $(M,
g_\infty, J_\infty)$ is holomorphic with respect to $J_\infty$. The readers can refer to \cite{chen-lebrun-weber} for details.  
We shall sketch a proof as follows. 
 When $M\sim\mathbb{CP}^2\sharp3\overline{\mathbb{CP}^2}$, each of holomorphic curves $H, E_1, E_2, E_3$ is the fixed point set of the isometric action of some circle action of
2-torus, and so each is totally geodesic with respect to the
metrics along the Calabi flow. By looking at the corresponding
fixed points set of the limit action of circle subgroups, we can
find corresponding  totally geodesic 2-sphere in $(M, g_\infty, J_\infty)$ which are the limits of the image of these submanifolds. Moreover, these
limit 2-spheres are holomorphic with respect to $J_\infty$ and the
homological intersection numbers of these holomorphic spheres  do not vary. Namely, we have still three holomorphic $\mathbb{CP}^1$s with self-intersection $-1$ as the images of the original exceptional divisors $E_1, E_2, E_3$.  Thus, by blowing down the images of $E_1, E_2, E_3$ and applying the
classification of the complex surface, we conclude that $(M,
J_\infty)$ is  biholomorphic to $\mathbb{CP}^2$  blowup three generic points. So there exists a
diffeomorphism $\Psi$ such that $\Psi_{*} J=J_\infty$. So
$\Psi^{*}g_\infty$ is an extremal metric in the class $[\omega]$ for
$(M, J)$. 
\end{proof}

\section{Appendix}
If $(M, g, J)$ is an extK metric on a compact complex surface, the
Calabi functional takes the value
\begin{equation*}
\begin{split}\cC(g)&=\int_MR^2dg\\
&=\underline{R}^2\int_Mdg+\int_M(R-\underline{R})^2dg\\
&=32\pi^2\frac{(c_1\cdot [\omega])^2}{[\omega]^2}+\|\cF\|^2,
\end{split}
\end{equation*}
where $\|\cF\|$ is the norm of the Calabi-Futaki invariant of a Mabuchi-Futaki invariant metric \cite{Futaki-Mabuchi}, see \cite{chen05} for the definition of this norm.  We define a functional $\cA([\omega])$ depending only the K\"ahler class
\begin{equation}\label{A-E-1}
\cA([\omega])=32\pi^2\frac{(c_1\cdot [\omega])^2}{[\omega]^2}+\|\cF\|^2.
\end{equation}
According to Chen \cite{chen05}, Donaldson \cite{Donaldson05}, $\cA([\omega])$ provides a natural lower bound of the Calabi functional and it is realized by an extK metric if there exists such a metric in the class $[\omega].$ 

A key ingredient of our result is to bound the Sobolev constant along the Calabi flow. In \cite{Tian}, \cite{Tian-Viaclovsky}, Tian made a shrewd observation that one can bound the Sobolev constants a priori for positive cscK metrics on K\"ahler surfaces with $c_1>0$ if the K\"ahler class satisfies
\begin{equation}\label{A-E-2}
c_1^2(M)-\frac{2}{3}\frac{(c_1(M)\cdot[\omega])^2}{[\omega]^2}>0.
\end{equation}
 This observation was generalized to extK metrics on K\"ahler surfaces with $c_1>0$ \cite{chen-weber} if the K\"ahler class satisfies the condition
\begin{equation}\label{A-E-3}
48\pi^2c_1^2(M)-\cA([\o])>0.
\end{equation}
Note that if the Futaki invariant is zero ($\|\cF\|=0$), \eqref{A-E-3} is reduced to \eqref{A-E-2}.   We define Tian's cone to be the set of K\"ahler classes which satisfy \eqref{A-E-2} and the generalized Tian's cone to be the set of K\"ahler classes which satisfy \eqref{A-E-3}. 

To generalize these ideas to the Calabi flow, we need that the Calabi energy is bounded by
\begin{equation}\label{A-E-4}
\int_MR^2dg<\cB([\omega]),
\end{equation}
where the functional $\cB([\omega])$ is defined to be
\begin{equation}\label{A-E-5}
\cB([\omega])=32\pi^2\left(c_1^2+\frac{1}{3}\frac{(c_1\cdot[\omega])^2}{[\omega]^2}\right)
+\frac{1}{3}\|\cF\|^2.
\end{equation}
Note that the Calabi functional has a natural lower bound $\cA([\omega])$. If \eqref{A-E-4} holds, we have 
\[
\cA([\omega])<\cB([\omega]).
\]
It is equivalent to  \[48\pi^2c_1^2>\cA([\omega]),\] which is exactly the generalized Tian's cone condition.\\

Consider the K\"ahler classes $[\omega]_x=3H-x(E_1+E_2+E_3)$ on $\mathbb{CP}^2\sharp 3\overline{\mathbb{CP}^2}$. It is not hard to check that
 $[\omega]_x$ satisfies the generalized Tian's cone condition for any $x\in (0, 3/2)$. Also one can compute
 \[
 \cB([\omega]_x)=192\pi^2+32\pi^2\frac{(3-x)^2}{3-x^2}. 
 \]
We can actually prove similar results as in Theorem \ref{T-2}.
\begin{theo}
Let $[\omega]_x=3H-x([E_1]+[E_2]+[E_3])$ on $\mathbb{CP}^2 \sharp 3 \overline{\mathbb{CP}^2}$. Suppose the initial metric $\omega_0\in [\omega]$ is invariant under the toric action and the action of $\mathbb{Z}_3$ and
\[
\int_MR^2dg_0<\cB([\omega]_x),
\] 
then the Calabi flow exists for all time with uniformly bounded curvature tensor. Moreover, it converges to a cscK metric in the class $[\omega]_x$ in the Cheeger-Gromov sense. 
\end{theo}
 
The strategy is similar as in the proof of Theorem \ref{T-2}.
All details can be found in the old version of our paper.
Note when $x=1$, the existence of K\"ahler-Einstein metric was shown by Siu \cite{Siu} and Tian-Yau \cite{TY}.
Recently Arezzo-Pacard \cite{AP} considered the existence of cscK metrics on blown ups. On $\mathbb{CP}^2\sharp 3\overline{\mathbb{CP}^2}$, their results imply  the existence of cscK metrics in $[\omega]_x$ for $x$ sufficiently close to $0$ or $3/2$.

\noindent Xiuxiong CHEN\\
xxchen@math.wisc.edu\\
Department of Mathematics\\
University of Wisconsin-Madison\\

\noindent Weiyong HE\\
whe@math.ubc.ca\\
Department of Mathematics\\
University of British Columbia\\
\noindent Current address:\\
whe@uoregon.edu\\
Department of Mathematics\\
University of Oregon

\end{document}